\documentclass[10pt,reqno]{amsart}
\usepackage{amsmath}
\usepackage{amsthm}
\usepackage{amssymb}
\usepackage{amsfonts}
\usepackage{latexsym}
\usepackage{hyperref}
\usepackage{cite}
\usepackage{xcolor}
\usepackage{enumitem}

\usepackage[font={small}]{caption}
\usepackage{enumitem}

\usepackage{graphicx}
\graphicspath{{./GRAPHS/}}
\usepackage{subcaption}

\setlist[enumerate]{leftmargin=*, label=(\alph*)}
\setlist[itemize]{leftmargin=*}

\theoremstyle{plain}
\newtheorem{theorem}{Theorem}
\newtheorem{corollary}[theorem]{Corollary}
\newtheorem{lemma}[theorem]{Lemma}

\theoremstyle{definition}

\newtheorem{remark}[theorem]{Remark}

\usepackage{tikz,ifthen}
\usetikzlibrary{intersections}
\usetikzlibrary{calc}
\usetikzlibrary{decorations.pathreplacing}
\usetikzlibrary{calc,decorations.markings}
\usepackage{comment}

\usetikzlibrary{positioning}
\usetikzlibrary{arrows,shapes,backgrounds,3d}
\usepackage{fix-cm}
\usetikzlibrary{decorations.pathreplacing,shapes,snakes}
\tikzstyle{box} = [rectangle, minimum width=1cm, minimum height=3cm,text centered, draw=black, fill=red!10]

\newcommand{\N}{\mathbb{N}}
\newcommand{\NN}{\mathcal{N}}

\newcommand{\ceiling}[1]{\lceil #1 \rceil}
\newcommand{\floor}[1]{\lfloor #1 \rfloor}

\renewcommand{\Re}{\operatorname{Re}}

\renewcommand{\pmod}[1]{\,\,(\operatorname{mod}#1)}

\let\oldenumerate=\enumerate
	\def\enumerate{
	\oldenumerate
	\setlength{\itemsep}{5pt}
	}
\let\olditemize=\itemize
	\def\itemize{
	\olditemize
	\setlength{\itemsep}{5pt}
	}

\allowdisplaybreaks

\begin{document}

\title[The error term in the truncated Perron formula]{The error term in the truncated Perron formula for the logarithm of an $L$-function}

	\author[S.~R.~Garcia]{Stephan Ramon Garcia}
	\address{Department of Mathematics and Statistics, Pomona College, 610 N. College Ave., Claremont, CA 91711, USA} 
	\email{stephan.garcia@pomona.edu}
	\urladdr{\url{http://pages.pomona.edu/~sg064747}}

    \author[J.~Lagarias]{Jeffrey Lagarias}
    \address{Department of Mathematics, University of Michigan, 530 Church Street, Ann Arbor, MI 48109, USA}
    \email{lagarias@umich.edu}
    \urladdr{\url{https://dept.math.lsa.umich.edu/~lagarias}}
           	
    \author[E.~S.~Lee]{Ethan Simpson Lee}
    \address{University of Bristol, School of Mathematics, Fry Building, Woodland Road, Bristol, BS8 1UG} 
    \email{ethan.lee@bristol.ac.uk}
    \urladdr{\url{https://sites.google.com/view/ethansleemath/home}}

\begin{abstract}
We improve upon the traditional error term in the truncated Perron formula for the logarithm of an $L$-function. All our constants are explicit.
\end{abstract}

\thanks{SRG supported by NSF Grant DMS-2054002. ESL
thanks the Heilbronn Institute for Mathematical Research for their support. We also thank the referee for their suggestions.}

\maketitle

\vspace{-20pt}
\section{Introduction}

The truncated Perron formula relates the summatory function of an arithmetic function to a contour integral that may be estimated using techniques from complex analysis. Let $F(s) = \sum_{n=1}^{\infty} f(n) n^{-s}$ be absolutely convergent on $\Re s > c_F$;
examples include the Riemann zeta-function, Dirichlet $L$-functions, 
the Dedekind zeta-function associated to a number field, and Artin $L$-functions. 
The truncated Perron formula tells us that if $x>0$ is not an integer, 
$T\geq 1$, and $c > c_F$, then
\begin{equation}\label{eqn:classicalTPF}
    \sum_{n \leq x} f(n) = \frac{1}{2\pi i }\int_{c-iT}^{c+iT} F(s) \frac{x^s}{s}\,ds + 
    O^*\bigg( \sum_{n=1}^{\infty} \bigg( \frac{x}{n} \bigg)^{c} |f(n)| \min\Big\{ 1, \frac{1}{T |\log \frac{x}{n}|}\Big\}\bigg),
\end{equation}
in which $O^*(g(x)) = h(x)$ means $|h(x)|\leq g(x)$; see \cite[Ex.~4.4.15]{MurtyAnalytic}, \cite[Ch.~7]{Koukoulopoulos}, \cite[Sect.~II.2]{Tenenbaum}, and \cite[Sect.~5.1]{MontgomeryVaughan}. We let $T$ depend on $x$ and let $c = c_F + 1/\log{x}$, so that $x^c = ex^{c_F}$. A variation of \eqref{eqn:classicalTPF} improves the order of the error term by truncating the integral at $\pm T^*$ for an unknown $T^*\in [T,O(T)]$ \cite{CullyHugillJohnston}, although this is inconvenient if one must avoid $T^*$ that correspond to the ordinates of non-trivial zeros of $F(s)$. The authors of \cite{CullyHugillJohnston} have also informed us in a personal communication that their paper inherited an unfortunate typo from another paper, so the error term in their variation of the truncated Perron formula could be worse by a factor of $\log{x}$; this means that our main result (Theorem \ref{Theorem:LogLog}) will be comparable in strength \textit{and} more straightforward to apply when compared against the outcome of their result.

For $\Re{s} > 1$, the logarithm of the Riemann zeta function $\zeta(s) = \sum_{n=1}^{\infty} n^{-s}$ is 
$\log \zeta(s) = \sum_{n=1}^{\infty} \Lambda(n)(\log n)^{-1}n^{-s}$, in which $\Lambda(n)$ 
is the von Mangoldt function.  The logarithm of a typical $L$-function is of the form
$\sum_{n=1}^{\infty} \Lambda(n)a_n(\log n)^{-1}n^{-s}$,
in which the $a_n$ are easily controlled.  For example, $|a_n| \leq 1$
for Dirichlet $L$-functions and $|a_n|\leq d$ for Artin $L$-functions of degree $d$; see \cite[Ch.~5]{iwaniec2004analytic}.
In these cases, the error term in \eqref{eqn:classicalTPF} with $c = 1/\log{x}$ is on the order of
\begin{equation}\label{eq:Traditional}
    \sum_{n=2}^{\infty} \Big( \frac{x}{n}\Big)^{\frac{1}{\log x}}  \frac{ \Lambda(n)}{n \log n}
    \min\Big\{ 1 , \frac{1}{T | \log \frac{x}{n} |} \Big\} 
    = O\bigg( \frac{\log x}{T} \bigg).
\end{equation}
Granville and Soundararajan used \eqref{eq:Traditional} with Dirichlet $L$-functions
to study large character sums \cite[(8.1)]{GranvilleSound}.
Cho and Kim applied it to Artin $L$-functions to obtain asymptotic bounds on Dedekind zeta residues \cite[Prop.~3.1]{ChoKim}.
A bilinear relative of \eqref{eq:Traditional} appears in Selberg's work on primes in short intervals \cite[Lem.~4]{Selberg}.  
Analogous sums arise with the logarithmic derivative of an $L$-function \cite[p.~106]{Davenport}, \cite[p.~44]{Patterson}.

We improve upon \eqref{eq:Traditional} asymptotically and explicitly.  Moreover, our result has a wide and explicit range of applicability.

\begin{theorem}\label{Theorem:LogLog}
If $x \geq 3.5$ is a half integer and $T \geq (\log \frac{3}{2})^{-1} > 2.46$, then
\begin{equation}\label{eq:MainSum}
    \sum_{n=2}^{\infty} \Big( \frac{x}{n}\Big)^{\frac{1}{\log x}}  \frac{ \Lambda(n)}{n \log n} \min\Big\{ 1 , \frac{1}{T | \log \frac{x}{n} |} \Big\} 
    \leq
    \frac{R(x)}{T},
\end{equation}
in which
\begin{equation}\label{eq:Rx}
    R(x) = 
    40.23 \log \log x
    +58.12
    +\frac{3.87}{\log x}
    +\frac{5.22 \log x}{\sqrt{x}}
    -\frac{1.84}{\sqrt{x}}.
\end{equation}
\end{theorem}

The following corollary employs \eqref{eqn:classicalTPF} with $T=x$ and $c = 1/\log{x}$. 
Since one can use analytic techniques to see the integral below is asymptotic to $\log{L(1,\chi)}$, 
one can relate $\log{L(1,\chi)}$ to a short sum.  We hope to do so explicitly in the future.

\begin{corollary}\label{cor:loglog}
Let $L(s, \chi)$ be an entire Artin $L$-function of degree $d$ such that 
\begin{equation*}
    L(s,\chi) = \prod_{p}\prod_{i=1}^{d} \left(1 - \frac{\alpha_i(p)}{p^s}\right)^{-1} 
    \quad \text{for $\Re s > 1$}
\end{equation*}
with $a(p^k) = \alpha_1(p)^k + \cdots + \alpha_d(p)^k$ for prime $p$.  Then with $R(x)$ as in \eqref{eq:Rx}
\begin{equation*}
    \sum_{1<n < x} \frac{\Lambda(n) a(n)}{n \log{n}} 
    = \frac{1}{2\pi i }\int_{\tfrac{1}{\log{x}}-ix}^{\tfrac{1}{\log{x}}+ix} \frac{x^s}{s} \log{L(1+s,\chi)}\,ds + 
    O^*\!\left( \frac{d\,R(x)}{x} \right).
\end{equation*}
\end{corollary}


\section{Preliminaries}\label{Section:Lemmas}
Here we establish several lemmas needed for the proof of Theorem \ref{Theorem:LogLog}.

\begin{lemma}\label{Lemma:Stieltjes}
    If $\sigma >0$, then $\log \zeta(1+\sigma) \leq - \log \sigma + \gamma \sigma$.
\end{lemma}

\begin{proof}
For $s> 1$, we have $\zeta(s) \leq e^{\gamma(s-1)}/(s-1)$ \cite[Lem.~5.4]{RamareExplicitDensity}.
Let $s = 1+\sigma$ and take logarithms to obtain the desired result.
\end{proof}

For real $z,w$, the equation $z = we^w$ can be solved for $w$ if and only if $z \geq -e^{-1}$.  There are two branches
for $-e^{-1} \leq z < 0$.  The lower branch defines
the \emph{Lambert $W_{-1}(z)$ function} \cite{Lambert}, 
which decreases to $-\infty$ as $z\to 0^-$; see Figure \ref{Figure:Lambert}.
For $n \geq 6>2e$, we define the strictly increasing sequence
\begin{equation}\label{eq:Yin}\qquad
y_n = \frac{-n}{2} W_{-1}\Big( \frac{-2}{n} \Big) \qquad \text{for $n \geq 6$}.
\end{equation}

\begin{lemma}\label{Lemma:Whine}
    For $n \geq 8$, we have $\frac{2y_n}{\log y_n} = n$ and $ y_n \geq \frac{n}{2} \log n$.
\end{lemma}

\begin{proof}
    For $n \geq 6$, the definition of $W_{-1}$ and \eqref{eq:Yin} confirm that $\frac{2y_n}{\log y_n} = n$.
    Thus, the desired inequality is equivalent to
    $W_{-1}(\frac{-2}{n} ) \leq -\log n$.
    Since $f(w) = we^w$ decreases on $(-\infty,-1]$ (Figure \ref{Figure:WeW}) and $-\frac{1}{e} < -\frac{2}{n} <0$, the desired inequality
    is equivalent to
    \begin{equation*}
    -\frac{2}{n} \geq f(-\log n) = (-\log n)e^{-\log n} = -\frac{\log n}{n},
    \end{equation*}
    which holds whenever $\log n \geq 2$.  This occurs for $n \geq e^2 \approx 7.38906$.
\end{proof}

\begin{remark}
For all $-e^{-1} \leq x < 0$, the bound $W_{-1}(x) \leq \log(-x) - \log(-\log(-x))$ is valid; see \cite[(8), (39)]{Loczi}. It follows from this observation and \eqref{eq:Yin} that
\begin{equation*}
    y_n 
    \geq \frac{n}{2} \left(\log\left(\frac{1}{2}\log\frac{n}{2}\right) + \log{n} \right),
\end{equation*}
which also implies Lemma \ref{Lemma:Whine} for $n\geq 15$.
\end{remark}

The next lemma is needed later to handle a few exceptional primes.

\begin{figure}
\centering
\begin{subfigure}[t]{0.47\textwidth}
\centering
\includegraphics[width=0.9\textwidth]{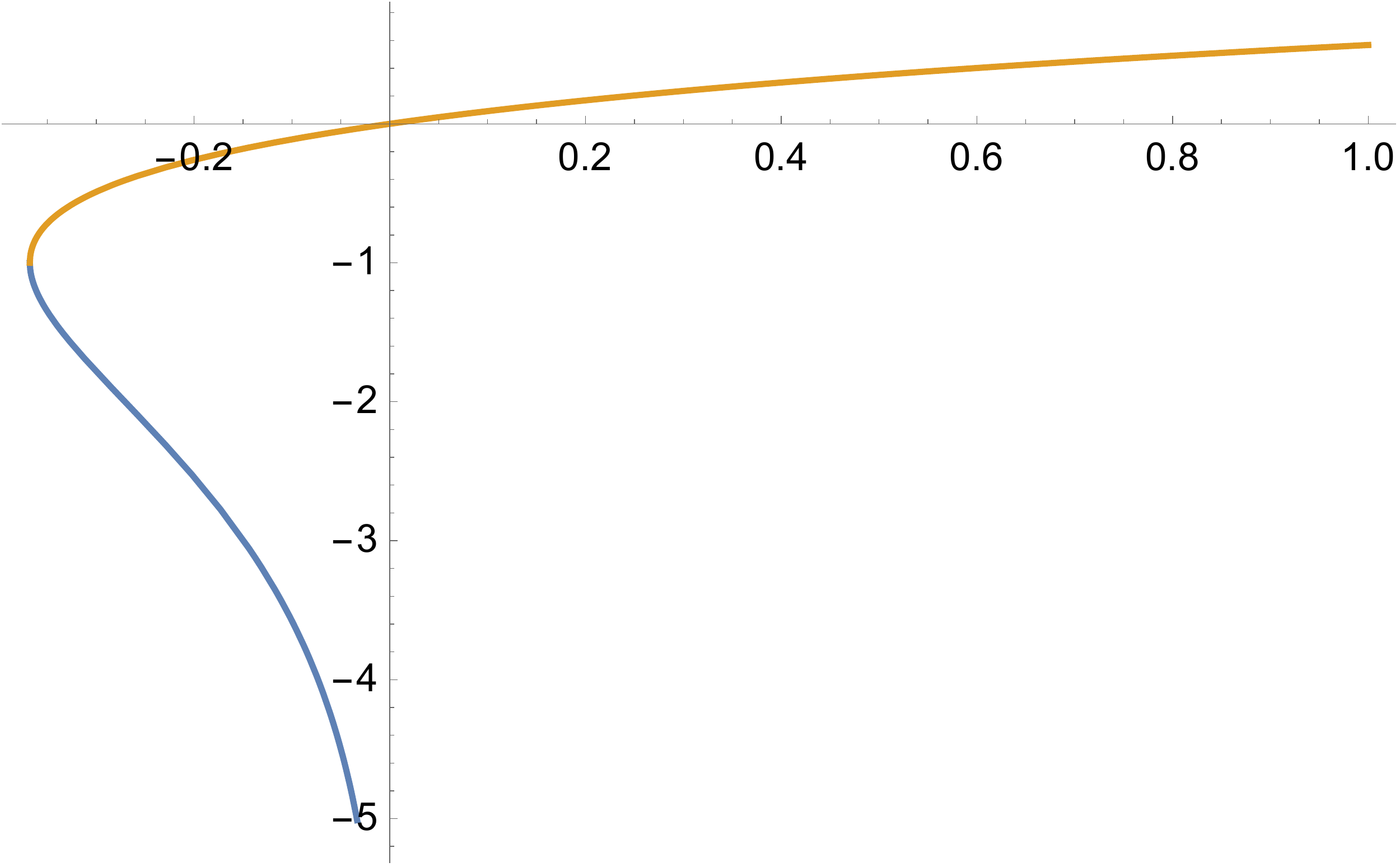}
\caption{The two branches of $z=we^w$.  The lower branch (blue) is the Lambert function $W_{-1}(z)$,
the upper branch (gold) is $W_0(z)$.}
\label{Figure:Lambert}
\end{subfigure}
\hfill
\begin{subfigure}[t]{0.47\textwidth}
\centering
\includegraphics[width=0.9\textwidth]{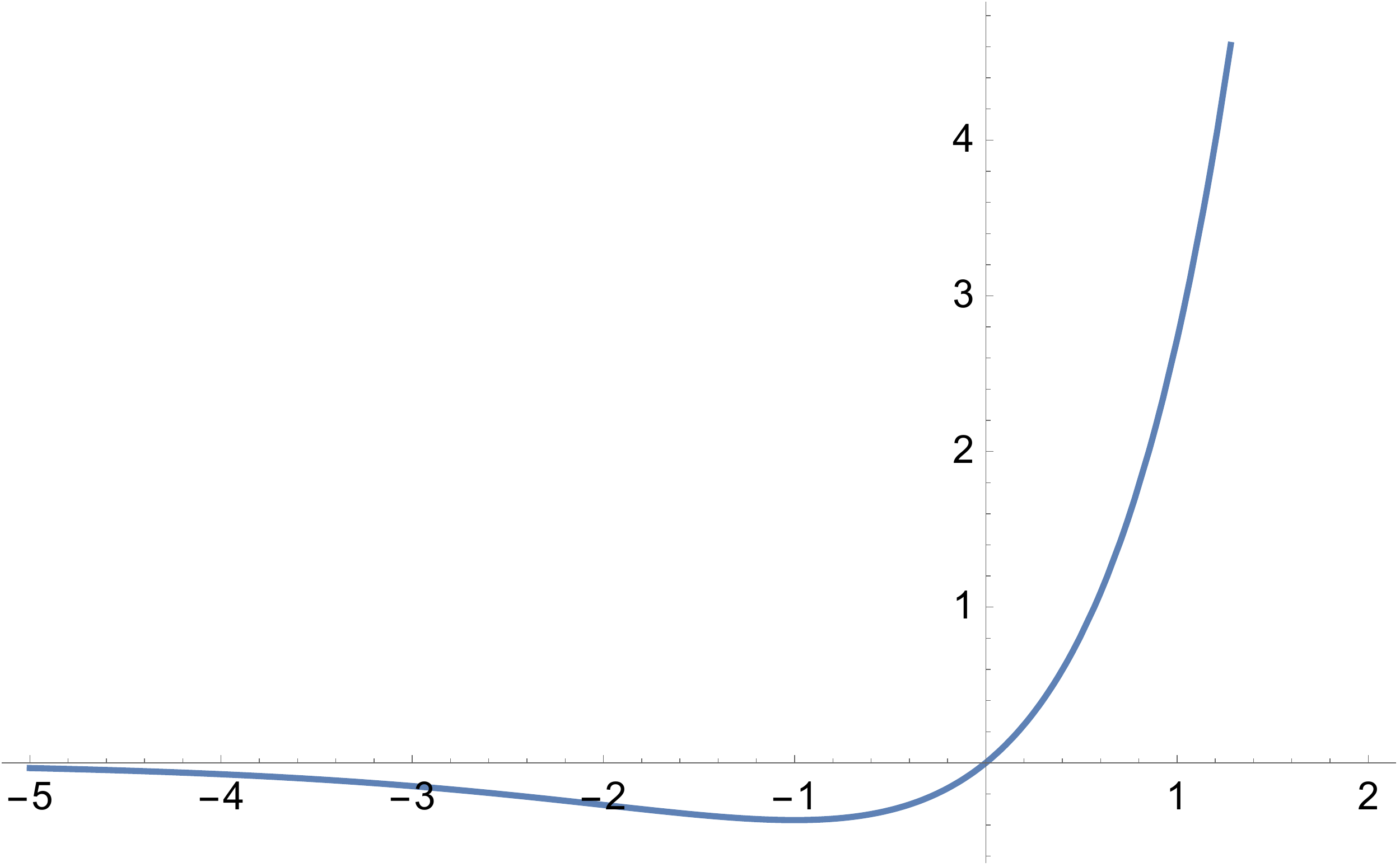}
\caption{$f(w) = we^w$ decreases on $(-\infty,-1]$.}
\label{Figure:WeW}
\end{subfigure}
\caption{Graphs relevant to the construction of the sequence $y_n$.}
\end{figure}

\begin{lemma}\label{Lemma:Sieve30}
    Let $x>1$ be a half integer and let $C = \frac{1284699552}{444215525}= 2.89206\ldots$.  
    \begin{enumerate}
        \item Let $p_{-8} < p_{-7} < \cdots < p_{-1} < x$ denote the largest eight primes (if they exist) in the interval $(\frac{x}{2},x)$.  
            We have the sharp bound\label{p:C1}
            \begin{equation}\label{eq:Mod30a}
                F_1(x) = \sum_{1 \leq n \leq 8}  \frac{1}{  x-p_{-n} } \leq  C;
            \end{equation}
            see Figure \ref{Figure:C1}.
            The corresponding summand in \eqref{eq:Mod30a} is zero if $p_{-n}$ does not exist.  
        
        \item Let $x<p_1 <p_2<\cdots < p_8$ denote the smallest eight primes (if they exist)
            in the interval $(x,\frac{3x}{2})$.  We have the sharp bound\label{p:C2}
            \begin{equation}\label{eq:Mod30b}
                F_2(x) = \sum_{1 \leq n \leq 8} \frac{1}{p_n -x} \leq  C;
            \end{equation}
		see Figure \ref{Figure:C2}.
            The corresponding summand in \eqref{eq:Mod30b} is zero if $p_{n}$ does not exist.  
    \end{enumerate}
\end{lemma}

\begin{figure}
    \centering
    \begin{subfigure}[t]{0.47\textwidth}
    \centering
    \includegraphics[width=0.9\textwidth]{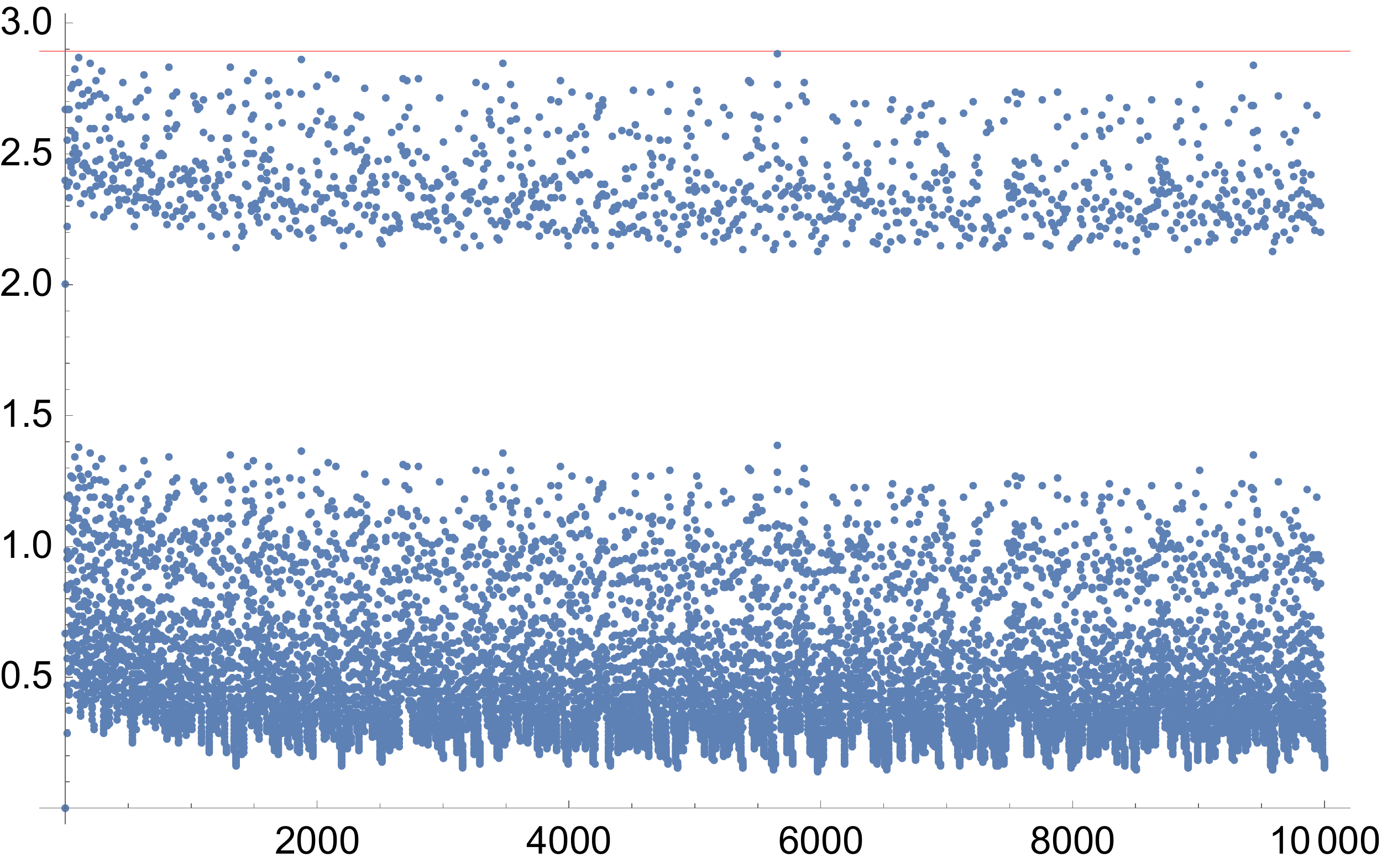}
    \caption{$F_1(x)$ is bounded above by $C$ (red).}
    \label{Figure:C1}
    \end{subfigure}
    \quad
    \begin{subfigure}[t]{0.47\textwidth}
    \centering
    \includegraphics[width=0.9\textwidth]{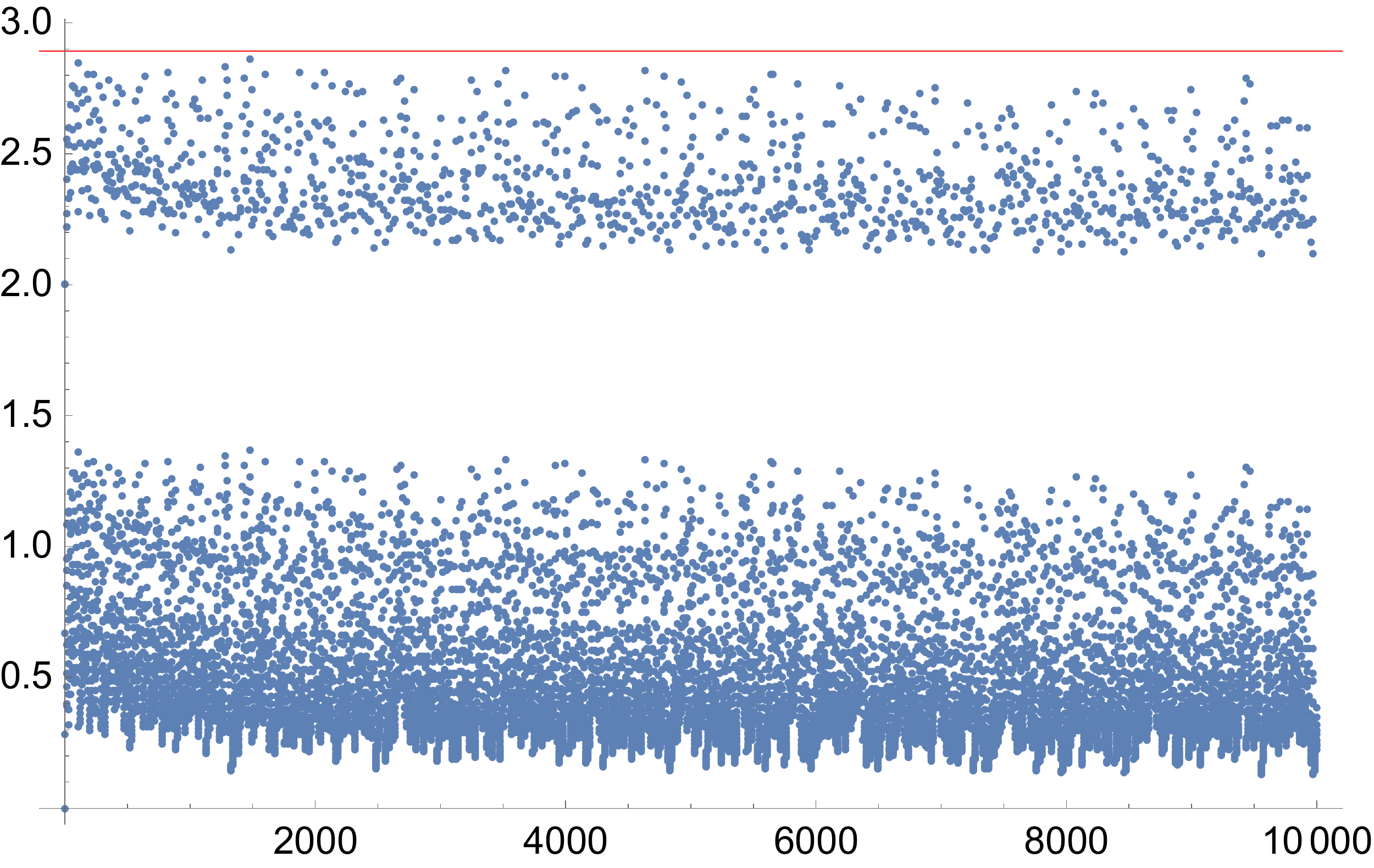}
    \caption{$F_2(x)$ is bounded above by $C$ (red).}
    \label{Figure:C2}
    \end{subfigure}
    \caption{The functions $F_1(x)$ and $F_2(x)$ behave erratically.}
    \label{Figure:C1C2}
\end{figure}

\vspace{-10pt}
\begin{proof}
    (a) If $x \geq 10.5$, then $2,3,5 \notin (\frac{x}{2},x)$.  Computation confirms that
    $F_1(x) \leq F_1(3.5) = \frac{8}{3} =2.66\ldots$ for $x \leq 9.5$.  Let $x \geq 10.5$.
    Then any prime in $(\frac{x}{2},x)$ is 
    congruent to one of $1, 7, 11, 13, 17, 19, 23, 29 \pmod{30}$.  
    There are finitely many patterns modulo $30$ that the $p_{-8},p_{-7},\ldots,p_{-1}$ may assume.
    Among these, computation confirms that $F_1(x)$ is maximized if 
    \begin{align*}
        p_{-1} &= \floor{x} \equiv 19\pmod{30},  & p_{-5} &= \floor{x}-12 \equiv 7\pmod{30},\\
        p_{-2} &= \floor{x}-2 \equiv 17\pmod{30},  & p_{-6} &= \floor{x}-18 \equiv 1\pmod{30},\\
        p_{-3} &= \floor{x}-6 \equiv 13\pmod{30},  & p_{-7} &= \floor{x}-20 \equiv 29\pmod{30},\\
        p_{-4} &= \floor{x}-8 \equiv 11\pmod{30},  & p_{-8} &= \floor{x}-26 \equiv 23\pmod{30},
    \end{align*}
    which yields the desired upper bound $C$.
    This prime pattern first occurs for $x = 88819.5$;
    see \url{https://oeis.org/A022013}.
        
    \medskip\noindent(b) If $x \geq 5.5$, then $2,3,5 \notin (x, \frac{3x}{2})$.  Observe that
    $F_2(x) \leq 2$ for $x \leq 4.5$ (attained at $x=1.5, 2.5, 4.5$).  Let $x \geq 5.5$.
    As in (a), any prime in $(x,\frac{3x}{2})$ is 
    congruent to one of $1, 7, 11, 13, 17, 19, 23, 29 \pmod{30}$.  
    Then $F_2(x)$ is maximized if 
    \begin{align*}
        p_{1} &= \ceiling{x} \equiv 11\pmod{30},  & p_{5} &= \ceiling{x}+12 \equiv 23\pmod{30},\\
        p_{2} &= \ceiling{x}+2 \equiv 13\pmod{30},  & p_{6} &= \ceiling{x}+18 \equiv 29\pmod{30},\\
        p_{3} &= \ceiling{x}+6 \equiv 17\pmod{30},  & p_{7} &= \ceiling{x}+20 \equiv 1\pmod{30}, \\
        p_{4} &= \ceiling{x}+8 \equiv 19\pmod{30}, & p_8 &= \ceiling{x} + 26 \equiv 7 \pmod{30}.
    \end{align*}
    which yields the desired upper bound $C$.
    Although this prime pattern occurs for $x=10.5$, not all eight primes lie in $(x,\frac{3}{2}x)$.    
    The first admissible value is $x = 15760090.5$; see \url{https://oeis.org/A022011}.
\end{proof}

We also need an elementary estimate on $k$th powers in intervals.

\begin{lemma}\label{Lemma:Powers}
Let $X > 1$ be a noninteger, $h>1$, and $k\geq 2$. 
\begin{enumerate}
\item There are at most $N_k + 1$ perfect $k$th powers in $[X,X+h)$, in which $N_k \leq \frac{h}{k \sqrt{X}}$.

\item The shortest gap between $k$th powers in $[X,X+h)$ (if they exist) is $G_k \geq k \sqrt{X}$.
\end{enumerate}
\end{lemma}

\begin{proof}
We may assume that $X$ is so large that $N_k \geq 1$.
Let $m = \ceiling{X^{\frac{1}{k}}}$ so that $m^k$ is the first $k$th power larger than $X$.
Consider the gaps $g_1,g_2,\ldots,g_{N_k}$ between the $N_k$ consecutive $k$th powers in $[X,X+h)$; see Figure \ref{Figure:UpperBoundN}. 
Then
\begin{equation*}
    G_k = \min\{ g_1,g_2,\ldots,g_{N_k}\}
	= g_1 = (m+1)^k - m^k 
	\geq km^{k-1}
	\geq k X^{\frac{k-1}{k}}
	\geq k \sqrt{X}.
\end{equation*}
The desired inequality follows since $N_k G_k \leq g_1 + g_2+ \cdots + g_{N_k} \leq h$.
\end{proof}

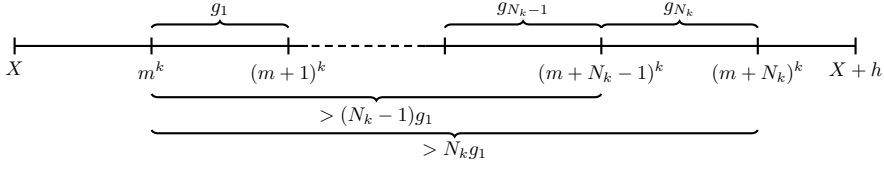
\begin{figure}
    \centering
    \begin{tikzpicture}[thick,scale=0.8,every node/.style={scale=0.75},xscale=-1.3]
    \draw(-10.75,0)--(-5.25,0);
    \draw[densely dashed](-5.25,0)--(-3.75,0);
    \draw(-3.75,0)--(0,0);
    \foreach \x/\xtext/\col in {-10.75/$X+h$/black, -9.5/$(m+N_k)^k$/black, -7.5/$(m+N_k-1)^k$/black, -5.5//black, -3.5/$(m+1)^k$/black, -1.75/$m^k$/black, 0/$X$/black}
        \draw[\col](\x,4pt)--(\x,-4pt) node[below] {\xtext};
    \draw[decorate, decoration={brace}, yshift=2ex]  (-7.5,0) -- node[above=0.4ex] {$g_{N_k}$}  (-9.5,0);
    \draw[decorate, decoration={brace}, yshift=2ex]  (-5.5,0) -- node[above=0.4ex] {$g_{N_k-1}$}  (-7.5,0);
    \draw[decorate, decoration={brace}, yshift=2ex]  (-1.75,0) -- node[above=0.4ex] {$g_{1}$}  (-3.5,0);
    \draw[decorate, decoration={brace}, yshift=2ex]  (-7.5,-1.1) -- node[below=0.4ex] {$>(N_k -1) g_1$}  (-1.75,-1.1);
    \draw[decorate, decoration={brace}, yshift=2ex]  (-9.5,-1.7) -- node[below=0.4ex] {$>N_k g_1$}  (-1.75,-1.7);
    \end{tikzpicture}
    \caption{Proof of Lemma \ref{Lemma:Powers}.}
    \label{Figure:UpperBoundN}
\end{figure}

Finally, we need an estimate on
the $n$th \emph{harmonic number} 
$H_n = \sum_{j=1}^n \frac{1}{j}$:
\begin{equation}\label{eq:TothMare}\qquad
    \frac{1}{2n+ \frac{2}{5}} < H_n - \log n - \gamma < \frac{1}{2n + \frac{1}{3}} \leq \frac{3}{7}
    \quad\text{for}\quad 
    n \geq 1,
\end{equation}
in which $\gamma$ is the Euler--Mascheroni constant \cite{TothMare}.   We require the upper bound
\begin{align}
\sum_{\ell=1}^n \frac{1}{2\ell - 1}
&= \Big(1 + \frac{1}{2} + \frac{1}{3} + \cdots + \frac{1}{2n}\Big) - \Big( \frac{1}{2} + \frac{1}{4} + \cdots + \frac{1}{2n} \Big)
= H_{2n} - \frac{1}{2} H_n  \nonumber \\
&\leq \bigg(\log 2n +\gamma + \frac{1}{4n + \frac{1}{3}} \bigg) - \frac{1}{2} \bigg( \log n + \gamma + \frac{1}{2n+\frac{2}{5}} \bigg) \nonumber \\
&\leq \frac{1}{2}\log n + \frac{1}{2}\gamma + \log 2 + \frac{1}{4n + \frac{1}{3}}- \frac{1}{4n+\frac{4}{5}} \nonumber \\
&= \frac{1}{2}\log n + \frac{1}{2}\gamma + \log 2 +  \frac{7}{240 n^2+68 n+4} \nonumber\\
&\leq \frac{1}{2}\log n + \frac{1}{2}\gamma + \log 2 +  \frac{7}{312}
\quad \text{for $n \geq 1$}. \label{eq:HarmSum}
\end{align}

\section{Proof of Theorem \ref{Theorem:LogLog}}\label{Section:Proof}

In what follows, $x\in \N+ \frac{1}{2} = \{ \frac{3}{2}, \frac{5}{2}, \frac{7}{2},\ldots\}$ and $c= \frac{1}{\log x}$.
Minor improvements below are possible; these were eschewed in favor of
a final estimate of simple shape.

\subsection{When $n$ is very far from $x$}\label{Subsection:Far}
Suppose that $n \leq \frac{x}{2}$ or $\frac{3x}{2} \geq n$. 
Then $\log \frac{x}{n} \leq -\log \frac{3}{2}$ or $\log \frac{3}{2} <\log 2 \leq \log \frac{x}{n}$, so $|\log \frac{x}{n}| \geq \log \frac{3}{2}$.
If $T \geq (\log \frac{3}{2})^{-1} > 2.46$, then
\begin{equation*}
    \min\Big\{ 1 , \frac{1}{T | \log \frac{x}{n}| } \Big\}  
    \leq \frac{1}{T  |\log \frac{x}{n}|} 
    \leq \frac{1}{T  \log \frac{3}{2}}.
\end{equation*}
For such $T$, the previous inequality and Lemma \ref{Lemma:Stieltjes} imply (recall that $c = \frac{1}{\log x}$)
\begin{align}
    \sum_{\substack{n \leq \frac{x}{2}\, \text{or}\\ n \geq  \frac{3x}{2}}}  \Big( \frac{x}{n}\Big)^c  \frac{ \Lambda(n) }{n \log n} \min\Big\{ 1 , \frac{1}{T | \log \frac{x}{n} |} \Big\}
    &\leq x^c \sum_{\substack{n \leq \frac{x}{2}\, \text{or}\\ n \geq  \frac{3x}{2} }} \frac{1}{n^{1+c}}\bigg( \frac{\Lambda(n)}{\log n}\bigg) \bigg(\frac{1}{T  \log \frac{3}{2} } \bigg)  \nonumber \\
    &\leq \frac{ x^c}{T\log \frac{3}{2}}  \log \zeta(1+c) \nonumber \\
    &\leq \frac{ x^c}{T \log \frac{3}{2}} ( - \log c + \gamma c) \nonumber \\
    &= \frac{e}{T \log \frac{3}{2}}\bigg(\log \log x + \frac{\gamma}{\log x} \bigg).
    \label{eq:VeryFar}
\end{align}

\subsection{Reduction to a sum over prime powers}\label{Subsection:ReducePP}
Suppose that $\frac{x}{2} < n < \frac{3x}{2}$.
Let $z = 1-\frac{n}{x}$ and observe that $|z| < \frac{1}{2}$.  Then 
\begin{equation*}
\log \frac{x}{n} = - \log(1-z) 
= z \left(- \frac{\log(1-z)}{z} \right),
\end{equation*}
in which the function in parentheses is positive and achieves its minimum value 
$2 \log \frac{3}{2} = 0.81093\ldots$ on $|z| < \frac{1}{2}$ at its left endpoint $-\frac{1}{2}$; see Figure \ref{Figure:LogFunctionMin}.
Then
\begin{equation}\label{eq:LogThingZed}
    | \log(1-z) | > \bigg(2 \log \frac{3}{2} \bigg) |z|
    \quad\text{for}\quad 
    |z| < \frac{1}{2},
\end{equation}
whose validity is illustrated in Figure \ref{Figure:LogFunctionWorks}.
     \begin{figure}
        \centering
        \begin{subfigure}[t]{0.47\textwidth}
        \centering
        \includegraphics[width=\textwidth]{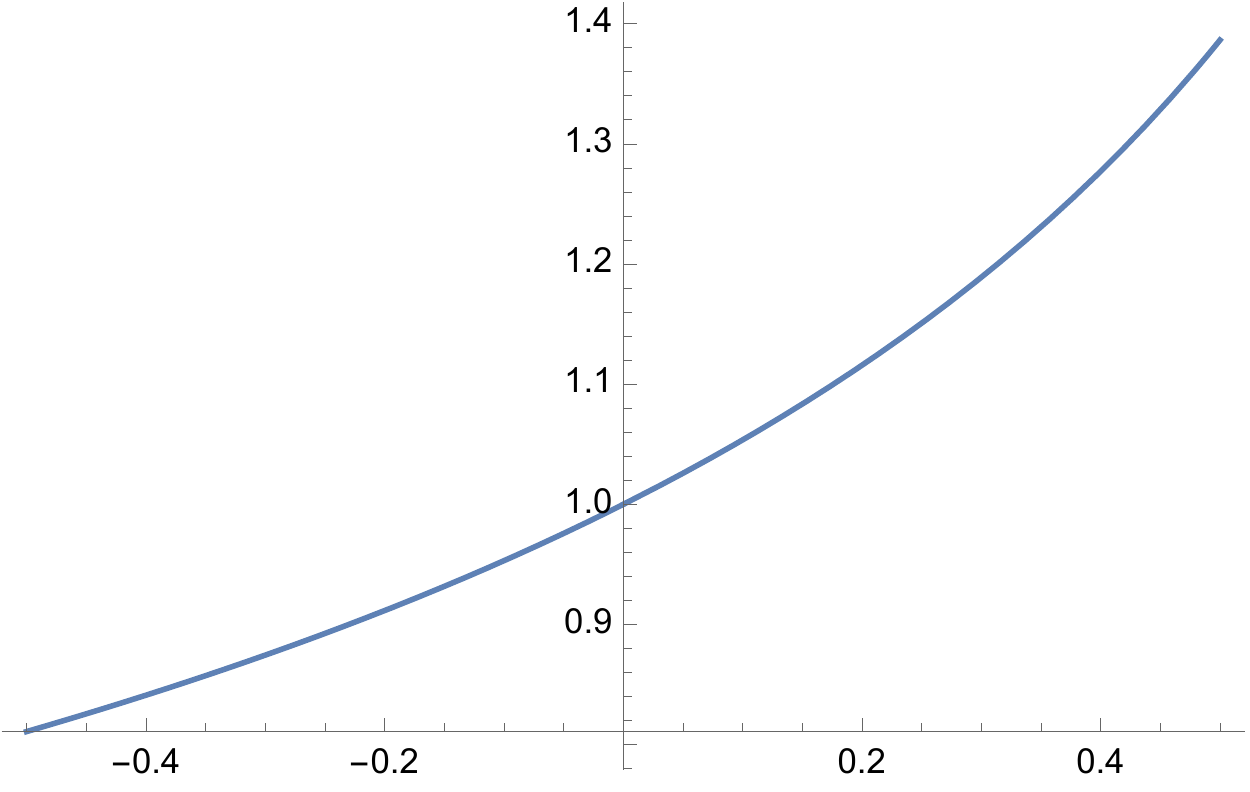}
        \caption{$- \frac{\log(1-z)}{z}$ on $|z| \leq \frac{1}{2}$.}
        \label{Figure:LogFunctionMin}
        \end{subfigure}    
        \hfill
        \begin{subfigure}[t]{0.47\textwidth}
        \centering
        \includegraphics[width=\textwidth]{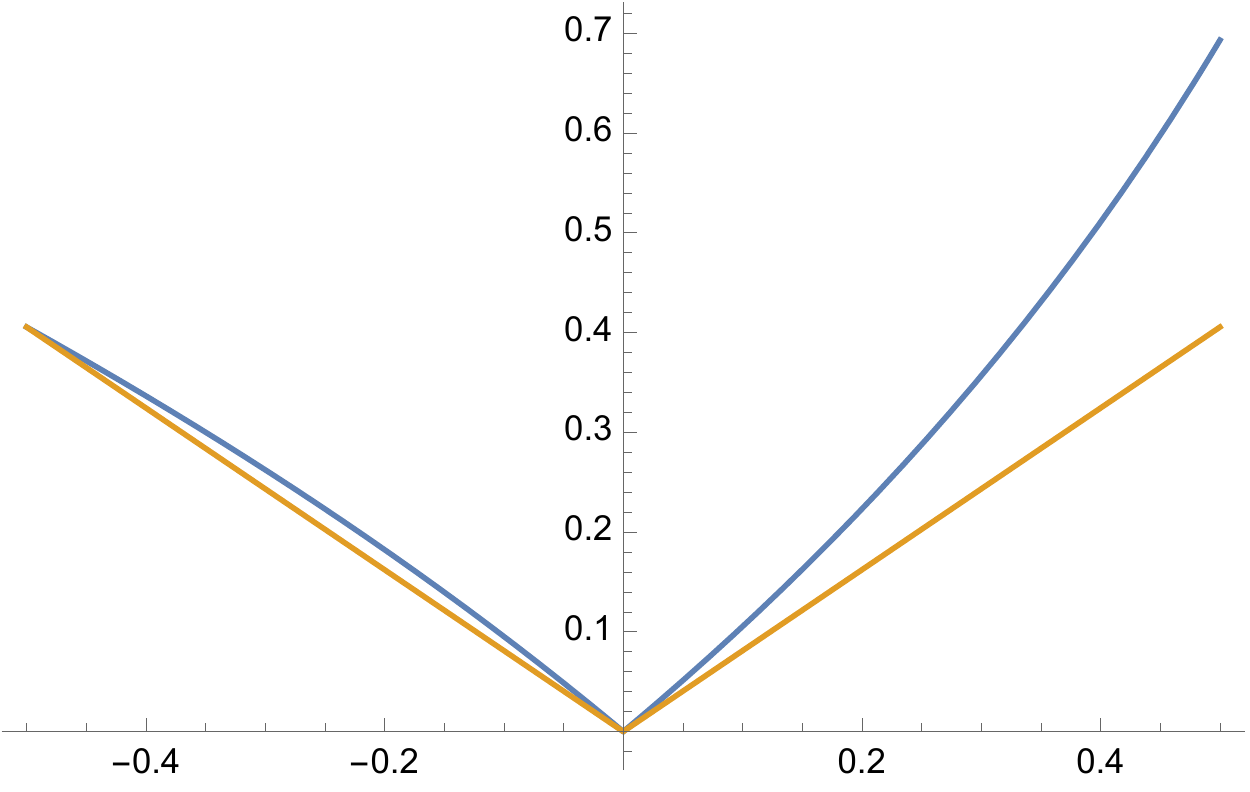}
        \caption{$|\log(1-z)|$ (blue) and $(2 \log \frac{3}{2}) |z|$ (gold) on $|z| \leq \frac{1}{2}$.}
        \label{Figure:LogFunctionWorks}
        \end{subfigure}    
        \caption{Graphs relevant to the derivation of \eqref{eq:LogThingZed}.}
    \end{figure}
Therefore,
\begin{equation}\label{eq:LogFunctionEstimateThing}
    \bigg| \log \frac{x}{n} \bigg| > \bigg( 2 \log \frac{3}{2}\bigg) \bigg|1 - \frac{n}{x}\bigg| 
    \quad\text{for}\quad
    \frac{x}{2} < n <  \frac{3x}{2},
\end{equation}
and hence
\begin{align}
& \sum_{\frac{x}{2} < n < \frac{3x}{2}}  \Big( \frac{x}{n}\Big)^c  \frac{ \Lambda(n) }{n \log n}
\min\Big\{ 1 , \frac{1}{T | \log \frac{x}{n} |} \Big\}  \nonumber \\
&\qquad\leq  \frac{x^c}{T}  \sum_{ \frac{x}{2} < n < \frac{3x}{2}}   \bigg( \frac{ \Lambda(n) }{n^{1+c} \log n}\bigg)
\bigg( \frac{1}{ | \log \frac{x}{n} |} \bigg) \nonumber \\
&\qquad\leq  \frac{x^c}{T} \sum_{\frac{x}{2} < n < \frac{3x}{2}} \bigg( \frac{ \Lambda(n) }{n^{1+c} \log n}\bigg)  \frac{1}{ (2\log \frac{3}{2}) | 1 - \frac{n}{x} |} 
&& (\text{by \eqref{eq:LogFunctionEstimateThing}}) \nonumber \\
&\qquad\leq \frac{ x^c}{T (2\log\frac{3}{2}) } \sum_{\frac{x}{2} < n < \frac{3x}{2}}  \frac{2}{x} \bigg( \frac{ \Lambda(n) }{n^c \log n}\bigg)\frac{1}{ |  1 - \frac{n}{x} |} 
&& (\text{since $\tfrac{x}{2} < n$}) \nonumber \\
&\qquad\leq \frac{x^c }{T \log \frac{3}{2}} \sum_{\frac{x}{2} < n < \frac{3x}{2}}  \bigg( \frac{ \Lambda(n) }{n^c \log n}\bigg)\frac{1}{ |  x-n |} 
\nonumber \\
&\qquad\leq \frac{x^c }{T \log\frac{3}{2}} \sum_{\frac{x}{2} < p^k < \frac{3 x}{2}}  \bigg( \frac{ \log p }{(p^k)^c k\log p}\bigg)\frac{1}{ |  x-p^k |} && (\text{def. of $\Lambda$})\nonumber \\
&\qquad\leq  \frac{e }{T \log\frac{3}{2}} \sum_{\frac{x}{2} < p^k < \frac{3x}{2}}  \frac{1}{k |  x-p^k |}, 
&&(\text{since $c=\tfrac{1}{\log x}$})
\label{eq:FirstPP}
\end{align}
in which the final two sums run over all prime powers $p^k$ in the stated interval.  

The remainder of the proof uses ideas from \cite[Lem.~2]{Goldston83} to estimate
\begin{equation}\label{eq:Trouble}
\sum_{\frac{x}{2} < p^k < \frac{3x}{2}}  \frac{1}{ k | x-p^k |}
= \underbrace{\sum_{\frac{x}{2} < p < \frac{3x}{2}}  \frac{1}{ | x-p |} }_{S_{\mathrm{prime}}(x)}
+ \underbrace{ \sum_{ \substack{\frac{x}{2} < p^k < \frac{3x}{2}\\ k\geq 2}}  \frac{1}{ k | x-p^k |} }_{S_{\mathrm{power}}(x)} .
\end{equation}


\subsection{The sum over primes}\label{Subsection:PrimeSum}
First observe that
\begin{equation*}
    S_{\mathrm{prime}}(x)
    \leq \underbrace{\sum_{\frac{x}{2} < p < x}  \frac{1}{  x-p } }_{S_{\mathrm{prime}}^-(x)}
    + \underbrace{\sum_{x < p < \frac{3x}{2}}  \frac{1}{ p-x}  }_{S_{\mathrm{prime}}^+(x)}.
\end{equation*}
We require the Brun--Titchmarsh theorem (see Montgomery--Vaughan \cite[Cor.~2]{MontgomeryVaughan}):
\begin{equation}\label{eq:Monty}
    \pi(X+Y) - \pi(X) \leq \frac{2Y}{\log{Y}}\quad\text{where $\pi(x) = \sum_{p\leq x}1$, $X > 0$, and $Y > 1$}.
\end{equation}

\subsubsection{The lower sum over primes}

Let $p_{-k} < p_{-(k-1)} < \cdots < p_{-2} < p_{-1}$
be the primes in $(\frac{x}{2}, x)$; note that $k \leq \frac{x}{2}$.
Apply \eqref{eq:Monty} with $X = x-y_n$ and $Y = y_n$ to get
\begin{equation*}\qquad
0 \leq \pi(x) - \pi(x-y_n) \leq \frac{2y_n}{\log y_n} = n \qquad \text{for $6 \leq n \leq k$}
\end{equation*}
by Lemma \ref{Lemma:Whine}, so $(x-y_n,x]$ contains at most $n$ primes.
Thus, $p_{-(n+1)} \leq x - y_n$ and
\begin{equation}\label{eq:xpy1}
\frac{1}{x-p_{-(n+1)}} \leq \frac{1}{y_n} \qquad \text{for $6 \leq n \leq k-1$}.
\end{equation}
Then Lemma \ref{Lemma:Whine}, which requires $k \geq 8$, and the integral test provide
\begin{align*}
\sum_{ \frac{x}{2} < p < x}  \frac{1}{  x-p }
&=\sum_{ 1 \leq n\leq 8} \frac{1}{  x-p_{-n} } + \sum_{9 \leq n \leq k}  \frac{1}{  x-p_{-n} } \\
&\leq F_1(x) + \sum_{8\leq n \leq k-1} \frac{1}{y_n} && (\text{by \eqref{eq:Mod30a} and \eqref{eq:xpy1}})\\
&\leq C + 2 \sum_{7 < n \leq \frac{x}{2}} \frac{1}{n \log n} && (\text{by Lemma \ref{Lemma:Whine}})\\
&\leq C +  2 \log \log x,
\end{align*}
which is valid for $k \leq 7$ since Lemma \ref{Lemma:Sieve30}a shows that the sum is majorized by $C$.

\subsubsection{The upper sum over primes}
Let $p_1<p_2< \cdots < p_k$ denote the primes in $(x, \frac{3x}{2})$ and note that $k \leq \frac{x}{2}$.
Then \eqref{eq:Monty} with $X = x$ and $Y = y_n$ ensures that
\begin{equation*}
0 \leq \pi(x+y_n) - \pi(x) \leq \frac{2y_n}{\log y_n} = n \qquad \text{for $6 \leq n \leq k$}
\end{equation*}
by Lemma \ref{Lemma:Whine}, so $(x, x+y_n]$ contains at most $n$ primes.  Thus,
$p_{n+1} \geq x + y_n$ and 
\begin{equation}\label{eq:xpy2}
\frac{1}{p_{n+1} - x} \leq \frac{1}{y_n} \qquad \text{for $6 \leq n \leq k$}.
\end{equation}
An argument similar to that above reveals that
\begin{equation*}
\sum_{x < p < \frac{3x}{2} } \frac{1}{p-x}
\leq \sum_{1\leq n \leq 8} \frac{1}{p_n -x} + \sum_{9\leq n \leq k} \frac{1}{p_n-x }
 \leq C +2 \log\log x.
\end{equation*}

\subsubsection{Final bound over primes}
For $x \in \N + \frac{1}{2}$, the previous inequalities yield
\begin{equation}\label{eq:PrimeBoundFinal}
S_{\mathrm{prime}}(x)
= \sum_{ \frac{x}{2} < p < \frac{3x}{2}}  \frac{1}{ | x-p |} 
\leq   2C +4 \log\log x.
\end{equation}

\subsection{The sum over prime powers}\label{Subsection:SumPP}
We now majorize
\begin{equation*}
S_{\mathrm{power}}(x) = \sum_{ \substack{ \frac{x}{2} < p^k <\frac{3x}{2} \\ k\geq 2}}  \frac{1}{ k | x-p^k |}.
\end{equation*}

\subsubsection{Initial reduction}
To bound $S_{\mathrm{power}}(x)$ it suffices to majorize
\begin{equation}\label{eq:SumEverything}
        S_{\mathrm{sqf}}(x) = \sum_{ \substack{ \frac{x}{2} < n^k <\frac{3x}{2}  \\ k\geq 2 \\ \text{$n \geq 2$ sq.~free} }}  \frac{1}{ k | x-n^k |},
\end{equation}
in which the prime powers $p^k$ are replaced with the powers $n^k$ of square free $n \geq 2$.  
The square-free restriction ensures that powers such as $2^6 = (2^2)^3 = (2^3)^2$ are not counted multiple times in \eqref{eq:SumEverything}. If
$\frac{x}{2} < n^k < \frac{3x}{2}$
and $k \geq 2$, then (since $n \geq 2$)
\begin{equation}\label{eq:BoundK}
k \leq  \frac{\log \frac{3x}{2} }{\log 2}  \leq \floor{2.4 \log x} \qquad \text{for $x \geq 3.5$}.
\end{equation}

\subsubsection{Nearest-power sets.}
The largest contributions to $S_{\mathrm{sqf}}(x)$ come from the powers closest to $x$.
We handle those summands separately and split the sum \eqref{eq:SumEverything} accordingly. 
For each $k \geq 2$, the inequalities
$\floor{ x^{\frac{1}{k}}}^k < x <  \ceiling{ x^{\frac{1}{k}}}^k$
exhibit the two $k$th powers nearest to $x$.  Define
\begin{equation}\label{eq:NNk}
    \NN_k \subseteq \big\{ \floor{ x^{\frac{1}{k}}}^k ,  \ceiling{ x^{\frac{1}{k}}}^k \big\}
\end{equation}
according to the following rules:
\begin{itemize}
    \item $\NN_k$ contains $\floor{ x^{\frac{1}{k}}}^k$ if it is square free and belongs to $( \frac{x}{2}, \frac{3x}{2} )$, and
    \item $\NN_k$ contains $\ceiling{ x^{\frac{1}{k}}}^k$ if it is square free and belongs to $( \frac{x}{2},\frac{3x}{2} )$. 
\end{itemize}
Consequently, $\NN_k$, if nonempty, contains only powers that satisfy the restrictions in \eqref{eq:SumEverything}.
The square-free condition ensures that $\NN_j \cap \NN_k = \varnothing$ for $j \neq k$.

Write $S_{\mathrm{sqf}}(x) = S_{\text{near}}(x) + S_{\text{far}}(x)$, in which
\begin{equation}\label{eq:SDef}
    S_{\text{near}}(x)
    \,\,= \!\!\!\! \sum_{ \substack{ \frac{x}{2} < n^k < \frac{3x}{2} \\ k \geq 2 \\ \text{$n \geq 2$ sq.~free} \\ n^k \in \NN_k}}  \frac{1}{ k | x-n^k |}
    \quad \text{and} \quad
    S_{\text{far}}(x)
    \,\,= \!\!\!\! \sum_{ \substack{ \frac{x}{2} < n^k < \frac{3x}{2} \\ k \geq 2 \\ \text{$n \geq 2$ sq.~free} \\ n^k \notin \NN_k}}  \frac{1}{ k | x-n^k |}.
    \end{equation}

\subsubsection{Near Sum.}  
For $x\geq 3.5$, a nearest-neighbor overestimate provides
\begin{align}
    S_{\text{near}}(x)
    &= \sum_{ \substack{ \frac{x}{2} < n^k < \frac{3x}{2} \\ k \geq 2 \\ \text{$n \geq 2$ sq.~free} \\ n^k \in \NN_k}}  \frac{1}{ k | x-n^k |}  && (\text{by \eqref{eq:SDef}}) \nonumber\\
    &= \sum_{k=2}^{\floor{ 2.4\log x} } \sum_{  m \in \NN_k }   \frac{1}{ k | x-m |}  &&  (\text{by \eqref{eq:BoundK}}) \nonumber\\
    &\leq \frac{1}{2} \sum_{k=2}^{\floor{ 2.4\log x} }  \sum_{  m \in \NN_k }   \frac{1}{  | x-m |}  && ( \text{since $k \geq 2$} ) \nonumber\\
    &\leq \frac{1}{2} \sum_{j=0}^{\floor{ 2.4\log x}-2 }   \left( \frac{1}{  x - (\floor{x}-j)}  + \frac{1}{ (\ceiling{x} +j) - x} \right)  && (\text{see below}) \nonumber\\
    &\leq \frac{1}{2} \sum_{\ell=1}^{\floor{ 2.4\log x}-1 }  \frac{2}{\ell - \frac{1}{2}} 
    \quad<\quad 2 \!\!\!\!\sum_{\ell=1}^{\floor{ 2.4\log x} }  \frac{1}{2\ell - 1}  \nonumber\\
    &\leq \log( \floor{2.4\log x}) + \gamma + 2\log 2 + \frac{14}{312}&& ( \text{by \eqref{eq:HarmSum}}) \nonumber\\
    &< \log \log x  + \gamma + 2\log 2  + \log 2.4 +  \frac{7}{156} . \label{eq:Near}
\end{align}
Let us elaborate on a crucial step above.  Consider the at most $\floor{2 \log x} - 1$ pairs of values $|x-m|$ that arise
as $m$ ranges over each $\NN_k$ with $2 \leq k \leq \floor{2 \log x}$ (since $\NN_j(x) \cap \NN_k = \varnothing$ for $j \neq k$, no $m$ appears more than once).  
Replace these values with the absolute deviations of $x$ from its 
$2 \times (\floor{2 \log x} - 1)$ nearest neighbors
$\floor{x}-j$ (to the left) and $\ceiling{x}+j$ (to the right), in which $0 \leq j \leq \floor{2 \log x}-2$.
Since $x \in \N+\frac{1}{2}$, these
deviations are of the form $\ell - \frac{1}{2}$ for $1 \leq \ell \leq \floor{2\log x}-1$.

\subsubsection{Splitting the second sum.} 

From \eqref{eq:SDef}, the second sum in question is
\begin{equation*}
    S_{\text{far}}(x)
    = \!\!\!\! \sum_{ \substack{ \frac{x}{2} < n^k < \frac{3x}{2} \\ k \geq 2 \\ \text{$n \geq 2$ sq.~free} \\ n^k \notin \NN_k}}  \frac{1}{ k | x-n^k |}
    \leq \!\!\!\! \sum_{ \substack{ \frac{x}{2} < n^k < \frac{3x}{2} \\ n,k \geq 2 \\ n^k \notin \NN_k}}  \frac{1}{ k | x-n^k |}
    = S_{\text{far}}^-(x)  + S_{\text{far}}^+(x),
\end{equation*}
in which
\begin{equation}\label{eq:FarSide}
    S_{\text{far}}^-(x)
    =  \sum_{k\geq 2} \sum_{ \substack{ \frac{x}{2} < n^k < x  \\ n \geq 2, n^k \notin \NN_k}}  \frac{1}{ k | x-n^k |}
    \quad \text{and} \quad
    S_{\text{far}}^+(x)
    =  \sum_{k\geq 2}  \sum_{ \substack{ x < n^k < \frac{3x}{2} \\ n \geq 2, n^k \notin \NN_k}}   \frac{1}{ k | x-n^k |}.
\end{equation}

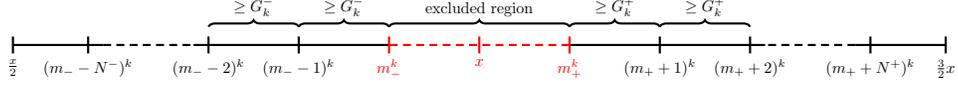
\begin{figure}
    \centering
    \begin{tikzpicture}[thick,scale=0.8,every node/.style={scale=0.6},xscale=-1]
    \draw(-7.75,0)--(-6.25,0);
    \draw[densely dashed](-6.25,0)--(-4.75,0);
    \draw(-4.75,0)--(-1.5,0);
    \draw[thick, red, densely dashed](-1.5,0)--(1.5,0);
    \draw(7.75,0)--(6.25,0);
    \draw[densely dashed](6.25,0)--(4.75,0);
    \draw(4.75,0)--(1.5,0);
    \foreach \x/\xtext/\col in {-7.75/$\frac{3}{2}x$/black, -6.5/$(m_+ +N^+)^k$/black, -4.5/$(m_+ +2)^k$/black, -3/$(m_+ +1)^k$/black, -1.5/$m_+^k$/red, 0/$x$/red,  1.5/$m_-^k$/red, 3/$(m_- -1)^k$/black, 4.5/$(m_- -2)^k$/black, 6.5/$(m_- -N^-)^k$/black, 7.75/$\frac{x}{2}$/black}
        \draw[\col](\x,4pt)--(\x,-4pt) node[below] {\xtext};
    \draw[decorate, decoration={brace}, yshift=2ex]  (1.5,0) -- node[above=0.4ex] {excluded region}  (-1.5,0);
    \draw[decorate, decoration={brace}, yshift=2ex]  (-1.5,0) -- node[above=0.4ex] {$\geq G_k^+$}  (-3,0);
    \draw[decorate, decoration={brace}, yshift=2ex]  (-3,0) -- node[above=0.4ex] {$\geq G_k^+$}  (-4.5,0);
    \draw[decorate, decoration={brace}, yshift=2ex]  (3,0) -- node[above=0.4ex] {$\geq G_k^-$}  (1.5,0);
    \draw[decorate, decoration={brace}, yshift=2ex]  (4.5,0) -- node[above=0.4ex] {$\geq G_k^-$}  (3,0);
    \end{tikzpicture}
    \caption{Analysis of $k$th powers in $[\tfrac{x}{2}, \tfrac{3x}{2}]$, in which 
    $m_- = \floor{x^{1/k}}$ and $m_+ = \ceiling{x^{1/k}}$ are excluded from consideration.  
    There are at most $N_k^-$ admissible $k$th powers in $[\frac{x}{2},x)$, with minimal gap size $G_k^-$,    
     and at most $N_k^+$ admissible $k$th powers in $[x,\frac{3x}{2})$, with minimal gap size $G_k^+$.}
    \label{Figure:Illustration}
\end{figure}
For $k \geq 2$,
Lemma \ref{Lemma:Powers} with $X =h= \frac{x}{2}$,
then with $X = x$ and $h = \frac{x}{2}$, implies that
\begin{equation}\label{eq:GGNN}
    G_k^- \geq \frac{k \sqrt{x}}{\sqrt{2}}, \quad
    N_k^- \leq \frac{\sqrt{x}}{2 \sqrt{2}}, \qquad
    \text{and} \qquad
    G_k^+ \geq k \sqrt{x}, \quad
    N_k^+ \leq \frac{ \sqrt{x}}{4}
\end{equation}
are admissible in Figure \ref{Figure:Illustration}.  For $1\leq j\leq N^-_k$ and $1\leq j\leq N^+_k$, respectively, 
\begin{equation*}
    |x - (m_- - j)^k| \geq \frac{j k \sqrt{x}}{\sqrt{2}} 
    \qquad\text{and}\qquad
    |x - (m_+ + j)^k| \geq j k \sqrt{x}.
\end{equation*}
Let $N_k^{\pm} \geq 1$, since otherwise the corresponding sum estimated below is zero.  Then
\begin{equation}\label{eq:HkMinus}
    \sum_{ \substack{\frac{x}{2} < n^k < x  \\ n \geq 2, n^k \notin \NN_k}}  \frac{1}{ k | x-n^k |}
    = \sum_{j=1}^{N^-_k}  \frac{1}{k |x-(m_- - j)^k|}
    \leq \frac{\sqrt{2}H_{N^-_k} }{k^2 \sqrt{x}}  
\end{equation}
and
\begin{equation}\label{eq:HkPlus}
    \sum_{\substack{x < n^k < \tfrac{3x}{2}  \\ n \geq 2, n^k \notin \NN_k}}  \frac{1}{ k | x-n^k |}
    = \sum_{j=1}^{N^+_k}  \frac{1}{k |x-(m_+ + j)^k|}
    \leq \frac{H_{N^+_k}}{k^2 \sqrt{x}}  .
\end{equation}
Therefore,
\begin{align}
S_{\text{far}}^-(x)
&=\sum_{k\geq 2} \sum_{ \substack{ \frac{x}{2} < n^k < x  \\ n \geq 2, n^k \notin \NN_k}}  \frac{1}{ k | x-n^k |} 
\leq  \frac{\sqrt{2}}{\sqrt{x}} \sum_{k\geq 2}\frac{H_{N^-_k} }{k^2 }  && (\text{by \eqref{eq:FarSide} and \eqref{eq:HkMinus}}) \nonumber\\
&\leq \frac{\sqrt{2}}{\sqrt{x}}   \bigg(\frac{1}{2}\log{x} - \frac{3}{2}\log{2} + \gamma + \frac{3}{7} \bigg) \sum_{k\geq 2}\frac{1}{k^2}
&& (\text{by \eqref{eq:TothMare} and \eqref{eq:GGNN}}) \nonumber \\
&=\frac{\pi^2-6}{3\sqrt{2 x}}  \bigg(\frac{1}{2}\log{x} - \frac{3}{2}\log{2} + \gamma + \frac{3}{7} \bigg) 
&&(\text{since $\zeta(2)-1 = \tfrac{\pi^2-6}{6}$}) \label{eq:FarMinus}
\end{align}
and
\begin{align}
S_{\text{far}}^+(x)
&=\sum_{k\geq 2}  \sum_{ \substack{ x < n^k < \frac{3x}{2} \\ n \geq 2, n^k \notin \NN_k}}   \frac{1}{ k | x-n^k |}
\leq \frac{1}{ \sqrt{x}}  \sum_{k \geq 2} \frac{H_{N^+_k}}{k^2} && (\text{by \eqref{eq:FarSide} and \eqref{eq:HkPlus}}) \nonumber\\
&\leq \frac{1}{ \sqrt{x}}  \bigg(\frac{1}{2}\log{x} - 2\log{2} + \gamma + \frac{3}{7} \bigg) \sum_{k \geq 2} \frac{1}{k^2}
&& (\text{by \eqref{eq:TothMare} and \eqref{eq:GGNN}}) \nonumber \\
&= \frac{\pi^2-6}{6\sqrt{x}}  \bigg(\frac{1}{2}\log{x} - 2\log{2} + \gamma + \frac{3}{7} \bigg) 
&&(\text{since $\zeta(2)-1 = \tfrac{\pi^2-6}{6}$}) \label{eq:FarPlus}
\end{align}

\subsubsection{Final prime-power estimate.}

Using \eqref{eq:Near}, \eqref{eq:FarMinus}, and \eqref{eq:FarPlus}, we can bound
\begin{equation*}
    S_{\mathrm{power}}(x)
    \leq S_{\mathrm{sqf}} (x)  = S_{\mathrm{near}}(x) + S_{\mathrm{far}}^-(x)  + S_{\mathrm{far}}^+(x) .
\end{equation*}
We postpone doing this explicitly until the finale below.

\subsection{Conclusion}\label{Subsection:Conclusion}
For $x \geq 3.5$, with $T \geq (\log \frac{3}{2})^{-1}$, the sum \eqref{eq:MainSum} is bounded by
\begin{align*}
    &\underbrace{\frac{e}{T \log \frac{3}{2}}\bigg(\log \log x + \frac{\gamma}{\log x} \bigg)}_{\text{by \eqref{eq:VeryFar}}}
    + \underbrace{\frac{e }{T \log\frac{3}{2}} \sum_{\frac{x}{2} < p^k < \frac{3x}{2}}  \frac{1}{k |  x-p^k |} }_{\text{by \eqref{eq:FirstPP}}} \\
    &\quad \leq \frac{e}{T \log \frac{3}{2}} \bigg(\log \log x + \frac{\gamma}{\log x} \bigg)
    + \frac{e}{T \log \frac{3}{2}} \underbrace{ \big(S_{\text{prime}}(x) + S_{\text{sqf}}(x) \big)}_{\text{by \eqref{eq:Trouble} and \eqref{eq:SumEverything}}} \\
    &\quad \leq \frac{e}{T \log \frac{3}{2}} \bigg(\log \log x + \frac{\gamma}{\log x} \bigg)
    + \frac{e}{T \log \frac{3}{2}}\bigg[ \underbrace{2C +4 \log\log x}_{\text{$S_{\text{prime}}(x)$ bounded by \eqref{eq:PrimeBoundFinal}}} \\
    &\qquad + 
    \underbrace{ \bigg(\log \log x  + \gamma + 2\log 2 + \log 2.4 +\frac{7}{156}  \bigg) }_{\text{$S_{\text{near}}(x)$ bounded by \eqref{eq:Near}}} 
    + \underbrace{\frac{\pi^2-6}{3\sqrt{2 x}}  \bigg(\frac{1}{2}\log{x} - \frac{3}{2}\log{2} + \gamma + \frac{3}{7} \bigg)}_{\text{$S_{\text{far}}^-$ bounded by \eqref{eq:FarMinus}}} \\
    &\qquad + \underbrace{\frac{\pi^2-6}{6\sqrt{x}}  \bigg(\frac{1}{2}\log{x} - 2\log{2} + \gamma + \frac{3}{7} \bigg)}_{\text{$S_{\text{far}}^+(x)$ bounded by \eqref{eq:FarPlus}}}
     \bigg] \\
    &\quad 
    < \frac{1}{T}\bigg(
    40.22465 \log \log x
    +58.11106
    +\frac{3.86972}{\log x}
    +\frac{5.21918 \log x}{\sqrt{x}}
    -\frac{1.85268}{\sqrt{x}}
    \bigg). \qed
\end{align*}


\bibliographystyle{amsplain}
\bibliography{PerronErrorTerm}

\providecommand{\bysame}{\leavevmode\hbox to3em{\hrulefill}\thinspace}
\providecommand{\MR}{\relax\ifhmode\unskip\space\fi MR }
\providecommand{\MRhref}[2]{%
  \href{http://www.ams.org/mathscinet-getitem?mr=#1}{#2}
}
\providecommand{\href}[2]{#2}
\begin{thebibliography}{10}

\bibitem{ChoKim}
P.~J. Cho and H.~H. Kim, \emph{Extreme residues of {D}edekind zeta functions},
  Math. Proc. Cambridge Philos. Soc. \textbf{163} (2017), no.~2, 369--380.
  \MR{3682635}

\bibitem{Lambert}
R.~M. Corless, G.~H. Gonnet, D.~E.~G. Hare, D.~J. Jeffrey, and D.~E. Knuth,
  \emph{On the {L}ambert {$W$} function}, Adv. Comput. Math. \textbf{5} (1996),
  no.~4, 329--359. \MR{1414285}

\bibitem{CullyHugillJohnston}
M.~Cully-Hugill and D.~R. Johnston, \emph{on the error term in the explicit
  formula of {R}iemann--von {M}angoldt}, Int. J. Number Theory (2022).

\bibitem{Davenport}
H.~Davenport, \emph{Multiplicative number theory}, third ed., Graduate Texts in
  Mathematics, vol.~74, Springer-Verlag, New York, 2000, Revised and with a
  preface by Hugh L. Montgomery. \MR{1790423}

\bibitem{Goldston83}
D.~A. Goldston, \emph{On a result of {L}ittlewood concerning prime numbers.
  {II}}, Acta Arith. \textbf{43} (1983), no.~1, 49--51. \MR{730847}

\bibitem{GranvilleSound}
A.~Granville and K.~Soundararajan, \emph{Large character sums}, J. Amer. Math.
  Soc. \textbf{14} (2001), no.~2, 365--397. \MR{1815216}

\bibitem{iwaniec2004analytic}
H.~Iwaniec and E.~Kowalski, \emph{Analytic number theory}, American
  Mathematical Society Colloquium Publications, vol.~53, American Mathematical
  Society, Providence, RI, 2004. \MR{2061214}

\bibitem{Koukoulopoulos}
D.~Koukoulopoulos, \emph{The distribution of prime numbers}, Graduate Studies
  in Mathematics, vol. 203, American Mathematical Society, Providence, RI,
  [2019] \copyright 2019. \MR{3971232}

\bibitem{Loczi}
L.~L\'{o}czi, \emph{Guaranteed- and high-precision evaluation of the {L}ambert
  {W} function}, Appl. Math. Comput. \textbf{433} (2022), Paper No. 127406, 22.
  \MR{4456416}

\bibitem{MontgomeryVaughan}
H.~L. Montgomery and R.~C. Vaughan, \emph{The large sieve}, Mathematika
  \textbf{20} (1973), 119--134. \MR{374060}

\bibitem{MurtyAnalytic}
M.~R. Murty, \emph{Problems in analytic number theory}, Graduate Texts in
  Mathematics, Springer New York, 2008.

\bibitem{Patterson}
S.~J. Patterson, \emph{An introduction to the theory of the {R}iemann
  zeta-function}, Cambridge Studies in Advanced Mathematics, vol.~14, Cambridge
  University Press, Cambridge, 1988. \MR{933558}

\bibitem{RamareExplicitDensity}
O.~Ramar\'{e}, \emph{An explicit density estimate for {D}irichlet
  {$L$}-series}, Math. Comp. \textbf{85} (2016), no.~297, 325--356.
  \MR{3404452}

\bibitem{Selberg}
A.~Selberg, \emph{On the normal density of primes in small intervals, and the
  difference between consecutive primes}, Arch. Math. Naturvid. \textbf{47}
  (1943), no.~6, 87--105. \MR{12624}

\bibitem{Tenenbaum}
G.~Tenenbaum, \emph{Introduction to analytic and probabilistic number theory},
  Cambridge Studies in Advanced Mathematics, vol.~46, Cambridge University
  Press, Cambridge, 1995, Translated from the second French edition (1995) by
  C. B. Thomas. \MR{1342300}

\bibitem{TothMare}
L.~T\'oth and S.~Mare, \emph{E 3432}, Amer. Math. Monthly \textbf{98} (1991),
  no.~3, 264.

\end{thebibliography}

\end{document}